\documentclass[12pt]{article}
	\usepackage{amsmath,fullpage,amsthm,amssymb,xcolor,graphicx}
	\newtheorem{prop}{Proposition}[section]
	\newtheorem{theorem}{Theorem}[section]
	\newtheorem{corollary}{Corollary}[section]
	\newtheorem{lemma}{Lemma}[section]
	
	\newtheorem{remark}{Remark}[section]
	\newtheorem{example}{Example}[section]

    \DeclareMathOperator{\nullity}{nullity}
	\DeclareMathOperator{\spec}{spec}
		\DeclareMathOperator{\diag}{diag}
		\DeclareMathOperator{\Rea}{Re}
			\DeclareMathOperator{\Ima}{Im}
				\DeclareMathOperator{\Tr}{Tr}
       
       \renewcommand{\det}{\operatorname{det}}
	\title{Spectral properties of distance Laplacian matrices of complex unit gain graphs}
\author{Aniruddha Samanta \thanks{Theoretical Statistics and Mathematics Unit, Indian Statistical Institute, Kolkata-700108, India. Email: aniruddha.sam@gmail.com}\ \and Deepshikha \thanks{Department of Mathematics, Shyampur Siddheswari Mahavidyalaya, University of Calcutta, West Bengal 711312, India. Email: dpmmehra@gmail.com }}

\date{\today}
\begin{document}
\maketitle
\baselineskip=0.25in

\begin{abstract}
A complex unit gain graph ($ \mathbb{T} $-gain graph), $ \Phi=(G, \varphi) $ is a graph where the function $ \varphi $ assigns a unit complex number to each orientation of an edge of $ G $, and its inverse is assigned to the opposite orientation. In this article, we study several spectral properties of distance Laplacian matrices of $\mathbb{T}$-gain graphs. In particular, we establish a characterization for the balanced $ \mathbb{T}$-gain graph in terms of the nullity of gain distance Laplacian matrices. As an example, it is shown that two switching equivalent $ \mathbb{T} $-gain graphs need not imply that their distance Laplacian spectra are the same. However, we provide a necessary condition for which two switching equivalent $ \mathbb{T} $-gain graphs have the same distance Laplacian spectra. Furthermore, we present a lower bound for spectral radii of gain distance Laplacian matrices in terms of the winner index. In addition, we establish some upper bounds for spectral radii of gain distance Laplacian matrices and characterize the equalities.   
\end{abstract}


\section{Introduction}
Throughout the paper, we consider $ G $ as a undirected connected simple graph with vertex set $ V(G)=\{ v_1, v_2, \dots, v_n\} $ and edge set $ E(G) $. If two vertices $ v_i $ and $ v_j $ are connected by an edge, then we write $ v_i\sim v_j $. If $ v_i\sim v_j $, then the  edge between them is denoted by $ e_{i,j} $. The \emph{distance} between two vertices $ v_i $ and $ v_j $ in $ G $ is denoted by $ d_G(v_i,v_j) $ ( or simply $ d(v_i,v_j) $) and is defined as the length of a shortest path between $ v_i $ and $ v_j $. The \emph{distance matrix} of $ G $ is an $ (n \times n) $ symmetric matrix $ \mathcal{D}(G) $ with $ (i,j) $th entry is $ d(v_i,v_j) $. The \emph{transmission} of a vertex $ v_i $ is defined by $ \Tr(v_i):=\sum\limits_{j=1}^{n}d(v_i,v_j) $. The \emph{transmission} matrix of $ G $ is a diagonal matrix $ \Tr(G):=\diag(\Tr(v_1), \Tr(v_2), \cdots, \Tr(v_n)) $. Then the \emph{distance Laplacian} and \emph{ distance signless Laplacian matrix} of $ G $ are $ \mathcal{DL}(G):=\Tr(G)-\mathcal{D}(G) $ and $ \mathcal{DQ}(G):=\Tr(G)+\mathcal{D}(G) $, respectively.  

In 2013, M Aouchiche et.al.\cite{Aouchiche_dis_Lap_1} introduce the notion of Laplacian and signless Laplacian for distance matrices of a simple graph $ G $ and discussed their spectral properties. Afterwards much progress has been made about distance Laplacian and distance signless Laplacian matrices for a simple graph $ G $, see \cite{Milan_dis_Lap_2,Huiqiu_dis_Lap_3,Aouchiche_dis_Lap_4,Das_dis_Lap_7,Das_dis_Lap_9,Fernandes_dis_Lap_5,Lin_dis_Lap_8} and references therein. A signed graph on an underlying graph $ G $ is a graph such that the sign of each edge of $ G $ is either $ 1 $ or $ -1 $ and it is denoted by $ \Sigma:=(G, \sigma) $, where $ \sigma:E(G)\rightarrow \{ -1, +1\} $. In 2020, Roshni et. al. \cite{hameed} presented a concept of two types of distance Laplacian matrices of signed graphs.

An oriented edge from the vertex $ v_i $ to $ v_j $ is denoted by $ \vec{e}_{i,j} $. Each undirected edge $ e_{i,j} $ determines a pair of oriented edges $ \vec{e}_{i,j} $ and $ \vec{e}_{j,i} $. An oriented edge set of  $ G $ is $ \vec{E}(G):=\{ \vec{e}_{i,j}, \vec{e}_{j,i}: e_{i,j}\in E(G)\} $. Let $ \mathbb{T}:=\{ z\in \mathbb{C}: |z|=1\} $. A \emph{complex unit gain graph} ( or $ \mathbb{T} $-gain graph) on an underlying graph $ G $ is denoted by $ \Phi=(G, \varphi) $, where $ \varphi:\vec{E}(G)\rightarrow \mathbb{T} $ is a mapping such that $\varphi(\vec{e}_{j,i})=\overline{\varphi(\vec{e}_{i,j})} $, for all $ e_{i,j}\in E(G) $. The \emph{adjacency matrix} of  a $ \mathbb{T} $-gain graph $ \Phi $ is a Hermitian matrix $ A(\Phi) $ and its $ (i,j) $th entry is defined by $ \varphi(\vec{e}_{i,j}) $, if $ v_i \sim v_j $ and zero otherwise. For more properties about Complex unit gain graphs and gain graphs, see \cite{Zaslav,reff1, Our-paper-1,A_alpha} and references therein.

An \emph{oriented path} in a $ \mathbb{T} $-gain graph $ \Phi $ from the vertex $ v_i $ to $ v_j $ is denoted by $ v_iPv_j $ and the gain of the path, $ \varphi(v_iPv_j) $ is the product of the gain of the oriented edges of the path. In fact, $ \varphi(v_jPv_i)=\overline{\varphi(v_iPv_j)}$. Let $ \vec{C}:v_1\rightarrow v_2 \rightarrow \cdots \rightarrow v_t \rightarrow v_1 $ be an oriented cycle. Then $ \varphi(\vec{C}) $ is the product of the consecutive oriented edges in $ \vec{C} $. If $ \vec{C}^* $ is the opposite oriented cycle of $ \vec{C} $, then   $ \overline{\varphi(\vec{C})}=\varphi(\vec{C}^*) $. If $ \varphi(\vec{C})=1$, then the cycle $ C $ is known as \emph{neutral} in $ \Phi $. A $ \mathbb{T} $-gain graph $ \Phi $ is called balanced if its all cycles are neutral.

Since some special graphs such as simple graph, oriented graph, signed graph and mixed graphs etc. can be obtained from a $ \mathbb{T} $-gain graph $ \Phi=(G, \varphi) $ by taking some specific $ \varphi $, so distance Laplacian matrices of a $ \mathbb{T} $-gain graph is also worth studying.

In \cite{gain_distance}, authors introduced gain distance matrices and study their properties. Later in \cite{GDL}, authors defined gain distance Laplacian matrices canonically. In this article, we study some spectral properties of two gain distance Laplacian matrices $ \mathcal{DL}^{\max}_{<}(\Phi) $ and $ \mathcal{DL}^{\min}_{<}(\Phi) $ of a $ \mathbb{T} $-gain graph $ \Phi=(G, \varphi) $ associated with a vertex ordering $ < $ on $ V(G) $. These concepts generalises the concepts of distance Laplacian matrices of simple graphs and signed distance Laplacian matrices of a signed graphs. A $ \mathbb{T} $-gain graph $ \Phi $ is distance compatible if and only if $ \mathcal{DL}^{\max}_{<}(\Phi)=\mathcal{DL}^{\min}_{<}(\Phi)$ with respect to any vertex ordering $ < $. We present that $ \Phi=(G, \varphi) $ is balanced if and only if $ \Phi $ is distance compatible and the distance Laplacian matrices of $ \Phi $ and $ G $ are co-spectral (Theorem \ref{th3.4}). Next, we show that the nullity of any gain distance Laplacian matrices is either $ 1 $ or $0$ and characterize each cases in Theorem \ref{th3.5}. The \emph{Winner index} of a graph $ G $ is $ W(G):=\sum\limits_{i>j} d(v_i,v_j)$. We obtain a lower bound of spectral radius of gain distance Laplacian matrices in terms of $ W(G) $ and $ n $ (Theorem \ref{th5.1}). Two $ \mathbb{T} $-gain graphs  $ \Phi=(G, \varphi) $ and $ \Psi=(G, \psi) $ are switching equivalent, write as $ \Phi \sim \Psi $ if there is a switching function $ \zeta: V(G)\rightarrow \mathbb{T} $ such that $ \varphi(\vec{e}_{i,j})=\zeta(v_i)^{-1}\psi(\vec{e}_{i,j}) \zeta(v_j)$ for all $ e_{i,j}\in E(G) $. We illustrate an example to show that if $ \Phi \sim \Psi $ need not imply that the spectrum of their gain distance Laplacian are the same (Example \ref{ex4.1}). However, if $ \Phi $ is distance compatible, then gain distance Laplacian matrices of $ \Phi $ and $ \Psi $ are co-spectral. Finally, we present two upper bound of spectral radius of gain distance Laplacian matrices and  characterize the equality (Theorem \ref{th4.9}, Theorem \ref{th5.4}).
\section{Definitions, notation and preliminary results}\label{prelim}

Let $ \Phi=(G, \varphi) $ be a connected $ \mathbb{T} $-gain graph on $ G $.  For  $ s,t\in V(G) $, $sPt$ denotes an oriented path from the  vertex $s$ to the vertex $t$.  Define three sets of paths $ \mathcal{P}(s,t), \mathcal{P}^{\min}(s,t) $ and $ \mathcal{P}^{\max}(s,t) $ as follows:
$$ \mathcal{P}(s,t)=\left\{ sPt: sPt \text{ is a shortest path} \right\};$$
$$ \mathcal{P}^{\min}(s,t)=\left\{ sPt \in \mathcal{P}(s,t):   \Rea(\varphi(sPt))=\min\limits_{s\tilde{P}t \in  \mathcal{P}(s,t)}\Rea(\varphi(s\tilde{P}t)) \right\}; $$
$$ \mathcal{P}^{\max}(s,t)=\left\{ sPt \in \mathcal{P}(s,t):   \Rea(\varphi(sPt))=\max\limits_{s\tilde{P}t \in  \mathcal{P}(s,t)}\Rea(\varphi(s\tilde{P}t)) \right\}.$$

We denote $ \left(V(G), <\right)$ as an ordered vertex set, where `$<$' is a total ordering of the vertices of $ G $. An ordering `$<_r$' is  the \emph{reverse ordering} of `$<$' if $ v_i<_r v_j $ if and only if $ v_j<v_i $, for any $ i,j $. An ordering `$<$' is  the \emph{standard vertex ordering} if  $ v_1<v_2<\dots <v_n $.

Let $ \Phi=(G, \varphi) $ be a connected $ \mathbb{T} $-gain graph with an ordered vertex set $ (V(G), <) $.  The \textit{maximum auxiliary gain function} with respect to the vertex ordering $ < $ is the map $ \varphi_{\max}^{<}: V(G)\times V(G) \rightarrow \mathbb{T} $ defined by the following properties:  $ \varphi_{\max}^{<}(s,s)=0 $ for all $ s\in V(G) $;
 whenever $ s<t $, $ \varphi_{\max}^{<}(s,t)=\varphi(sPt) $, where $ sPt $ is the oriented path from $ s $ to $ t $ such that $ sPt \in \mathcal{P}^{\max}(s,t) $ with $ \Ima(\varphi(sPt))=\max\limits_{s\tilde{P}t \in \mathcal{P}^{\max}(s,t)}\Ima(\varphi(s\tilde{P}t)).$ Moreover, $ \varphi_{\max}^{<}(t,s)=\overline{ \varphi_{\max}^{<}(s,t)}$.
		 The \textit{minimum auxiliary gain function} with respect to the vertex ordering $ < $ is the map $ \varphi_{\min}^{<}: V(G)\times V(G) \rightarrow \mathbb{T} $ defined by the following properties:	
			 $ \varphi_{\min}^{<}(s,s)=0 $ for all $ s\in V(G) $;
			 whenever $ s<t $, $ \varphi_{\min}^{<}(s,t)=\varphi(sPt) $, where $ sPt $ is the oriented path from $ s $ to $ t $ such that $ sPt \in \mathcal{P}^{\min}(s,t) $ with $ \Ima(\varphi(sPt))=\min\limits_{s\tilde{P}t \in \mathcal{P}^{\min}(s,t)}\Ima(\varphi(s\tilde{P}t)).$ Moreover, $ \varphi_{\min}^{<}(t,s)=\overline{ \varphi_{\min}^{<}(s,t)}$. Note that,  for $s< t$,  $ \varphi^{<}_{\max}(s, t)  (\mbox{~resp.,~} \varphi^{<}_{\min}(s, t)  )$ is the maximum (resp., minimum) gain of the shortest paths from $s$ to $t$ with respect to the lexicographic ordering of complex numbers ( $ a+ib\leq c+id $ with respect to the lexicographic ordering  if  $ a<c $ or $ a=c $ and $ b\leq d $).  That is,  
\begin{enumerate}
	\item[(1)] If $ s<t $, then $\varphi_{\max}^{<}(s,t)=\max\limits_{sPt\in \mathcal{P}(s,t)} \varphi(sPt)  $. Further $ \varphi_{\max}^{<}(t,s)=\overline{ \varphi_{\max}^{<}(s,t)}$.
	\item[(2)] If  $ s<t $, then  $\varphi_{\min}^{<}(s,t)=\min\limits_{sPt\in \mathcal{P}(s,t)} \varphi(sPt)  $. Further $ \varphi_{\min}^{<}(t,s)=\overline{ \varphi_{\min}^{<}(s,t)}$.
\end{enumerate}
Here the maximum and minimum are taken with respect to the lexicographic ordering of complex numbers.
For any two vertices $ s,t\in V(G)$, there are two gain distances from the vertex $ s $ to the vertex $ t $  which are defined as follows:	 $d^{<}_{\max}(s,t)=\varphi^{<}_{\max}(s,t)d(s,t)$ and  	$ d^{<}_{\min}(s,t)=\varphi^{<}_{\min}(s,t)d(s,t).$
 The \textit{gain distance matrices} $ \mathcal{D}_{<}^{\max}(\Phi) $ and $ \mathcal{D}_{<}^{\min}(\Phi) $ associated with $ < $ are defined as follows:	
			$ \mathcal{D}_{<}^{\max}(\Phi)= \left( d^{<}_{\max}(v_i, v_j)\right)$,
			 $ \mathcal{D}_{<}^{\min}(\Phi)= \left( d^{<}_{\min}(v_i, v_j) \right)$.
	 Here $d^{<}_{\max}(v_i, v_j)$ and $d^{<}_{\min}(v_i, v_j)$ are the $(i, j)$th entry of $ \mathcal{D}_{<}^{\max}(\Phi)$ and $ \mathcal{D}_{<}^{\min}(\Phi)$, respectively.
A $ \mathbb{T} $-gain graph $ \Phi=(G, \varphi) $ is  vertex ordering independent (simply, ordering independent), if $ \mathcal{D}^{\max}_{<}(\Phi)=\mathcal{D}^{\max}_{<_r}(\Phi)$ and $ \mathcal{D}^{\min}_{<}(\Phi)=\mathcal{D}^{\min}_{<_r}(\Phi) $, where $  < $ is the standard vertex ordering on $V(G)$. In this case, we define $\mathcal{D}^{\max}(\Phi):=\mathcal{D}^{\max}_{<}(\Phi)=\mathcal{D}^{\max}_{<_r}(\Phi)$ and $ \mathcal{D}^{\min}(\Phi):= \mathcal{D}^{\min}_{<}(\Phi)=\mathcal{D}^{\min}_{<_r}(\Phi)$.
A $ \mathbb{T} $-gain graph $ \Phi=(G, \varphi) $ is called gain distance compatible (simply, distance compatible) if $ \mathcal{D}^{\max}_{<}(\Phi)=\mathcal{D}^{\min}_{<}(\Phi) $, where $  < $ is the standard vertex ordering. In this case, we define $ \mathcal{D}(\Phi):=\mathcal{D}^{\max}_{<}(\Phi)=\mathcal{D}^{\min}_{<}(\Phi)$. 

\begin{lemma}[{\cite[Proposition 3.2]{gain_distance}}]\label{Th0.1}
	Let $ \Phi=(G, \varphi) $ be a distance compatible $ \mathbb{T} $-gain graph. If $ \Phi \sim \Psi $, then $ \Psi $ is distance compatible and $ \spec(D(\Phi))=\spec(D(\Psi)) .$
\end{lemma}

		Let $ \Phi=(G, \varphi) $ be a $ \mathbb{T} $-gain graph with a vertex ordering $ < $. The complete $ \mathbb{T} $-gain graph with respect to $\mathcal{D}^{\max}_{<}(\Phi) $, denoted by $ K^{\mathcal{D}^{\max}_{<}}(\Phi) $, is a complete $ \mathbb{T} $-gain graph of $ V(G) $ vertices with edge gain of $ \vec{e}_{i,j} $ is  $\varphi^{<}_{\max}(v_i,v_j)$ for all $v_i,v_j \in V(G)$. Similarly
		$ K^{\mathcal{D}^{\min}_{<}}(\Phi) $ is defined  using  $\mathcal{D}^{\min}_{<}(\Phi) $.

If the $ \mathbb{T} $-gain graph $ \Phi $ is ordering independent, then define $K^{\mathcal{D}^{\max}}(\Phi):= K^{\mathcal{D}^{\max}_{<}}(\Phi)=K^{\mathcal{D}^{\max}_{<_r}}(\Phi)  $ and $K^{\mathcal{D}^{\min}}(\Phi):= K^{\mathcal{D}^{\min}_{<}}(\Phi)=K^{\mathcal{D}^{\min}_{<_r}}(\Phi)  $. If $ \Phi $ is distance compatible then it is ordering independent, and   $ K^{\mathcal{D}^{\max}}(\Phi)=K^{\mathcal{D}^{\min}}(\Phi)  $. Then define $ K^{\mathcal{D}}(\Phi):= K^{\mathcal{D}^{\max}}(\Phi)=K^{\mathcal{D}^{\min}}(\Phi)$.

\begin{theorem}[{\cite[Theorem 5.2]{gain_distance}}]\label{Th1.1}
	Let $ \Phi=(G, \varphi) $ be a $ \mathbb{T} $-gain graph with vertex ordering $<$. Then the following statements are equivalent.
	\begin{enumerate}
		\item [(i)] $ \Phi $ is balanced.
		\item [(ii)] $ K^{\mathcal{D}^{\max}}(\Phi) $ is balanced.
		\item [(iii)] $ K^{\mathcal{D}^{\min}}(\Phi) $ is balanced.
		\item [(iv)] $ \mathcal{D}^{\max}(\Phi) =\mathcal{D}^{\min}(\Phi)$ and associated complete $ \mathbb{T} $-gain graph $ K^{\mathcal{D}}(\Phi) $ is balanced.
	\end{enumerate}
\end{theorem}

Let $ \mathbb{C}^{m\times n} $ denote the set of all $ m\times n $ matrices with complex entries. For $ A=(a_{ij})\in \mathbb{C}^{n \times n} $, define $ |A|=(|a_{ij}|) $.

\begin{theorem}[{\cite[Theorem 8.1.18]{horn-john2}}]\label{th3}
	Let $ M,N \in \mathbb{C}^{n \times n} $ be two matrices. If $ M \geq |N| $, then $ \rho(M)\geq \rho(|N|) \geq \rho(N) $.	
\end{theorem}

\begin{theorem}[{\cite[Theorem 8.4.5]{horn-john2}}]\label{th2}
	Let $ M,N \in \mathbb{C}^{n \times n}$. Suppose $  M$ is irreducible and non-negative and $ M \geq |N| $. Let $ \mu=e^{i\theta} \rho(N)$ be a given maximum modulus eigenvalue of $ N $. If $ \rho(M)=\rho(N) $, then there is a unitary diagonal matrix $ D $ such that $ N=e^{i\theta}DMD^{-1} $.
\end{theorem}

\begin{lemma}[{\cite[Lemma 2.1]{Ellingham}}]\label{lm2.1}
	Let $ A $ be an $ n \times n $ real symmetric matrix. If $ \mu $ is an eigenvalue of $ A $ associated with an eigenvector $ x $ whose all entries are non negative, then
	$$  \min_{1\leq j \leq n}S_j(A)\leq \mu \leq \max_{1\leq j \leq n}S_j(A),$$ 
	where $ S_j(A) $ is the sum of the entries of $ j$th row of $ A $.
\end{lemma}

Let $ M $ be a square matrix. Then $ \spec(M) $ denotes the collection of all eigenvalues of $ M $ with the corresponding multiplicities. The spectral radius of $ M $ is defined by $ \rho(M):=\max\limits_{\lambda \in \spec (M)}|\lambda|$. 

\begin{theorem} [{\cite[Theorem 2.14]{Berman}}]\label{th4}
	Let $ M$ and $N $ be  square matrices. If $ M \geq |N| $ and $ M $ is irreducible, then
	\begin{enumerate}
		\item [(1)] For every eigenvalue $ \lambda $ of $ N $, $ \rho(M)\geq |\lambda| $.
		\item[(2)] Equality holds in (1) if and only if $ N=e^{i\theta} UMU^*$, where $ e^{i\theta}=\frac{\lambda}{\rho(M)}$ and $ |U|=I $.
	\end{enumerate} 
\end{theorem}

\section{Nullity of gain distance Laplacian matrices}
We begin this section with the concepts of distance Laplacian matrices of complex unit gain graph, or simply, gain distance Laplacian matrices. A gain distance Laplacian matrix can be identified with a Laplacian matrix of a positively weighted $ \mathbb{T} $-gain graph. In this section, we characterize a balanced $ \mathbb{T} $-gain graph in terms of gain distance Laplacian spectrum. Then we obtain that the nullity of gain distance Laplacian matrices is either $ 1 $ or $0$. We show that a $ \mathbb{T} $-gain graph $ \Phi $ is balanced if and only if the nullity of its distance Laplacian matrix is $ 0 $.

Let $ \Phi=(G, \varphi) $ be a connected $ \mathbb{T} $-gain graph with an ordered vertex set $ (V(G), <) $, where $ V(G)=\{ v_1, v_2, \dots, v_n \}$ is the vertex set of $ G $. For a vertex $ v_s\in V(G) $, the \emph{transmission} of $ v_s $, denoted by $ \Tr(v_s) $, is defined as $ \Tr(v_s):=\sum\limits_{v_t\in V(G)}d(v_s,v_t) $.  The \emph{transmission matrix} of a graph $ G $ is a diagonal matrix defined as  $ \Tr(G):= \diag(\Tr(v_1), \Tr(v_2), \dots, \Tr(v_n)) $.  Then \emph{the distance Laplacian matrices of complex unit gain graph $\Phi$ } or simply, \textit{the gain distance Laplacian matrices} of $ \Phi $ associated with the ordering $ < $ are defined as follows: 
\begin{itemize}
\item[(1)] $ \mathcal{DL}^{\max}_{<}(\Phi):=\Tr(G)-\mathcal{D}^{\max}_{<}(\Phi) $.
\item[(2)] $ \mathcal{DL}^{\min}_{<}(\Phi):=\Tr(G)-\mathcal{D}^{\min}_{<}(\Phi) $.
\end{itemize}

If a $\mathbb{T} $-gain graph $ \Phi $ is ordering independent, then all the maximum (resp., minimum) gain distance Laplacian matrices are equal and the equal matrices is defined as $\mathcal{DL}^{\max}(\Phi):=\Tr(G)-\mathcal{D}^{\max}(\Phi) $ (resp., $ \mathcal{DL}^{\min}(\Phi):=\Tr(G)-\mathcal{D}^{\min}(\Phi) $). For a distance compatible $\mathbb{T}$-gain graph $\Phi$, \textit{the gain distance Laplacian matrix} is defined as $\mathcal{DL}(\Phi):= \mathcal{DL}^{\max}(\Phi) = \mathcal{DL}^{\min}(\Phi).$\\

\textit{The maximum and minimum gain distance signless Laplacian matrices} of a connected $ \mathbb{T} $-gain graph $ \Phi=(G, \varphi) $ associated with an ordered vertex set $ (V(G), <) $ are defined as follows:  $  \mathcal{DQ}^{\max}_{<}(\Phi):=\Tr(G)+\mathcal{D}^{\max}_{<}(\Phi) $ and
		 $ \mathcal{DQ}^{\min}_{<}(\Phi):=\Tr(G)+\mathcal{D}^{\min}_{<}(\Phi) $, respectively. 
If $ \Phi $ is vertex ordering independent, then  define $\mathcal{DQ}^{\max}(\Phi):=\mathcal{DQ}^{\max}_{<}(\Phi)$ and $\mathcal{DQ}^{\min}(\Phi):=\mathcal{DQ}^{\min}_{<}(\Phi)$, where $ < $ the standard vertex ordering.  For distance compatible $\mathbb{T}$-gain graph $ \Phi $,  \textit{the gain distance signless Laplacian matrix} is  defined by $\mathcal{DQ}(\Phi):=\mathcal{DQ}^{\max}(\Phi)=\mathcal{DQ}^{\min}(\Phi).$\\

 Let $ \Phi=(G, \varphi) $ be a $ \mathbb{T} $-gain graph and $ w:E(G)\rightarrow \mathbb{R}^{+} $ be a weight function. A \textit{positively weighted $ \mathbb{T} $-gain graph} associated with $ \Phi $ and $ w $, denoted by $ \Phi_w=(G, \varphi, w) $ or simply $ \Phi_w=(G, \varphi_w) $ where $ \varphi_w:\vec{E}(G)\rightarrow \mathbb{C} $ is defined as $ \varphi_w(\vec{e}_{i,j})=\varphi(\vec{e}_{i,j})w(e_{i,j}).$ The adjacency matrix of $ \Phi_w $ is the Hermitian matrix $ A(\Phi_w) $ and its $ (i,j)th $ entry is $ \varphi_w(\vec{e}_{i,j}) $ if $ v_i \sim v_j $ and zero otherwise. The Laplacian matrix of a positively weighted $ \mathbb{T} $-gain graph $ \Phi_w $, denoted by $ L(\Phi_w) $, and is defined by $ L(\Phi_w):=D(\Phi_w)-A(\Phi_w), $
where $ D(\Phi_w) $ is a diagonal matrix whose $ (i,i)th $ entry is $ \sum\limits_{v_i \sim v_j}w(e_{i,j}) $. Note that a gain distance Laplacian matrix can be identified with a Laplacian of a positively weighted $ \mathbb{T} $-gain graph.\\

 Let $ G $ be a simple graph with $ m $ edges and $ n $ vertices. Let us consider an arbitrary but fixed orientation of $ G $. For an oriented edge $ \vec{e}=\overrightarrow{v_iv_j} $, define $ h(\vec{e}):=v_j $ and $ t(\vec{e}):=v_i $. The\textit{ incidence matrix of $ \Phi_w $} corresponding to an orientation of $ G $,  denoted by $ I(\Phi_w) $, is the $ n\times m $ matrix with the  $ (i,j)th $ entry is:
$$\eta_{v_i\vec{e}_j} = \left\{ \begin{array}{rcl}
\varphi(\vec{e}_{j})\sqrt{w(e_j)} & \mbox{for}
& t(\vec{e}_j)=v_i \\
 -\sqrt{w(e_j)} & \mbox{for} & h(\vec{e}_j)=v_i\\
0 &  \mbox{} & \text{otherwise.}
\end{array}\right.$$

For the sake of completeness, we include a couple of known results on spectral properties of gain distance Laplacian matrices.
\begin{theorem}[{\cite{GDL}}]\label{th3.1}
Let $ \Phi_w $ be a positively weighted $ \mathbb{T} $-gain graph. Then $ L(\Phi_w)=I(\Phi_w)I(\Phi_w)^{*}.$
\end{theorem}


A positively weighted $ \mathbb{T} $-gain graph $ \Phi_w=(G, \varphi, w) $ is  \textit{balanced} if  $\Phi=(G, \varphi)  $ is balanced. Now we consider $ G $ as a connected graph.
  
\begin{theorem}[{\cite{GDL}}]\label{th3.3}
  	Let $ \Phi_w=(G, \varphi, w) $ be a positively weighted connected $ \mathbb{T} $-gain graph. Then $ \Phi=(G, \varphi) $ is balanced if and only if $\det L(\Phi_w)=0$.
  \end{theorem}
Let us present the following equivalence between the balance of a  $ \mathbb{T} $-gain graph, the determinant of gain distance Laplacian matrices and the distance compatibility.  

\begin{theorem}[{\cite{GDL}}]\label{th3.4}
	Let $ \Phi = (G,\varphi) $ be a $ \mathbb{T} $-gain graph. Then the following are equivalent: 
	\begin{itemize}
		\item[(1)] $ \Phi $ is balanced.
		\item [(2)] $ \det \mathcal{DL}^{\max}(\Phi)=0 $.
		\item [(3)] $ \det \mathcal{DL}^{\min}(\Phi)=0 $.
		\item[(4)] $ \mathcal{DL}^{\max}(\Phi)= \mathcal{DL}^{\min}(\Phi)$ and $ \det \mathcal{DL}(\Phi)=0$.
	\end{itemize} 
\end{theorem}

\begin{theorem}\label{th3.4}
	A $ \mathbb{T} $-gain graph $ \Phi=(G, \varphi) $ is balanced if and only if $ \mathcal{DL}^{\max}(\Phi)=\mathcal{DL}^{\min}(\Phi)=\mathcal{DL}(\Phi) $ and $ \mathcal{DL}(\Phi) $ is co-spectral with $ \mathcal{DL}(G) $.
\end{theorem}
\begin{proof}
	Let $\Phi=(G,\varphi)$ be a balanced $ \mathbb{T} $-gain graph. Then there exist a switching function $ \zeta $ such that $ \Phi^\zeta =G$. Now by Theorem \ref{Th1.1}(iv), $ \mathcal{D}(\Phi) $ exist. Then $ \mathcal{DL}^{\max}(\Phi)=\mathcal{DL}^{\min}(\Phi)=\mathcal{DL}(\Phi) $. Also by the proof of the Lemma \ref{Th0.1}, there exist a diagonal unitary matrix $ P=diag(\zeta(v_1), \zeta(v_2), \dots, \zeta(v_n)) $, such that $ \mathcal{D}(\Phi)=P^{*}\mathcal{D}(G)P $. Then 
	\begin{align*}
		\mathcal{DL}(\Phi)&=Tr(G)-\mathcal{D}(\Phi)=Tr(G)-P^{*}\mathcal{D}(G)P=P^{*}(Tr(G)-\mathcal{D}(G))P=P^{*}\mathcal{DL}(G)P.
	\end{align*}
Therefore, $ \mathcal{DL}(\Phi) $ is co-spectral with $ \mathcal{DL}(G) $. Conversely, suppose $ \mathcal{DL}^{\max}(\Phi)=\mathcal{DL}^{\min}(\Phi)=\mathcal{DL}(\Phi) $ and $ \mathcal{DL}(\Phi) $ is co-spectral with $ \mathcal{DL}(G) $. Then $ \det \mathcal{DL}(\Phi)= \det \mathcal{DL}(G)=0 $. Therefore, by Theorem \ref{th3.4}, $ \Phi $ is balanced.
\end{proof}

\begin{prop} \label{prop.4.1}
	Let $ \Phi=(G, \varphi) $ be a connected $ \mathbb{T} $-gain graph with vertex ordering $ < $. Then the distance Laplacian matrices of $ \Phi $ associated with $ < $ are positive semi-definite.
\end{prop}
\begin{proof}
	Let $ \Phi=(G, \varphi) $ be a connected $ \mathbb{T} $-gain graph of $ n $ vertices with an ordered vertex set $ (V(G), <) $ . Let us consider a positively weighted $ \mathbb{T} $-gain graph $ \Psi_w=(\psi,w,K_n) $, where $ \psi(\vec{e}_{i,j})=\varphi_{\max}^{<}(v_i,v_j)$ and $ w(e_{i,j})=d_G(v_i,v_j) $, for any $ i\neq j $. Then $ L(\Psi_w)=\mathcal{DL}^{\max}_{<}(\Phi)$. By Theorem 3.1, $ L(\Psi_w)=I(\Psi_w)I(\Psi_w)^{*}$. Now, for any $ z\in \mathbb{C}^{n}\setminus \{0 \} $, 
	\begin{align*}
		zL(\Psi_w)z^*&=\frac{1}{||z||^2}zI(\Psi_w)z^*zI(\Psi_w)^*z^*\\
		&=\frac{1}{||z||^2} ||zI(\Psi_w)z^*||^2\geq 0.	
	\end{align*}
Hence $\mathcal{DL}^{\max}_{<}(\Phi)$ is positive semi-definite. Similarly we can  $ \psi(\vec{e}_{i,j})=\varphi_{\min}^{<}(v_i,v_j)$ and $ w(e_{i,j})=d_G(v_i,v_j) $, for any $ i\neq j $. Then $ \mathcal{DL}^{\min}_{<}(\Phi)$ is positive semi-definite.
\end{proof}

\begin{theorem}\label{th3.5}
	 Let $ \Phi=(G, \varphi) $ be a $ \mathbb{T} $-gain graph with a vertex ordering $ < $. Then one of the following statements holds.
	 \begin{enumerate}
	 	\item [(1)] $\nullity(\mathcal{DL}^{\max}_<(\Phi))=1$ if and only if $ \Phi $ is balanced.
	 	\item [(2)] $ \nullity(\mathcal{DL}^{\max}_<(\Phi))=0$ if and only if $ \Phi $ is unbalanced.
	 \end{enumerate}
\end{theorem}

\begin{proof}
	Let $ \Phi=(G, \varphi) $ be a connected $ \mathbb{T} $-gain graph of $ n $ vertices with an ordered vertex set $ (V(G), <) $ . Let us consider a positively weighted $ \mathbb{T} $-gain graph $ \Psi_w=(\psi,w,K_n) $, where $ \psi(\vec{e}_{i,j})=\varphi_{\max}^{<}(v_i,v_j)$ and $ w(e_{i,j})=d_G(v_i,v_j) $, for any $ i\neq j $. Then $ L(\Psi_w)=\mathcal{DL}^{\max}_{<}(\Phi)$. By Theorem 3.1, $ L(\Psi_w)=I(\Psi_w)I(\Psi_w)^{*}$.
	
	\noindent $ (1):$ Suppose $\nullity(\mathcal{DL}^{\max}_<(\Phi))=1$. Then $ \det \mathcal{DL}^{\max}_<(\Phi)=\det L(\Psi_w)=0 $. By Theorem 3.2, $ \Psi $ is balanced. Therefore, from the construction of $ \Psi $, $ \Phi $ is balanced. 
	Conversely, suppose $ \Phi $ is balanced. Then $ \Psi=(K_n,\psi) $ is balanced. Therefore, $ L(\Psi_w) $ and $ L((K_{n})_w) $ are similar. In fact, $ L((K_{n})_w)=I((K_{n})_w)I((K_{n})_w)^t$. Let $ X=(x_1,x_2, \dots, x_n)\in \mathbb{R}^{n}\setminus \{ 0\} $ be a column vector. Then $ X^tI((K_{n})_w)=0 $ implies $ x_i=x_j $ for all $ i\neq j $. Thus all the components of $ X $ are equal. Therefore, the dimension of left null space of $ I((K_{n})_w) $ is at most $ 1 $. Since the rows of $ I((K_{n})_w) $ are dependent, so the nullity of  $ I((K_{n})_w) $ is at least $1$. Thus $\nullity(I((K_{n})_w))=1$. Therefore, $  \nullity(\mathcal{DL}^{\max}_<(\Phi))=\nullity (L(\Psi_w))=\nullity (I((K_{n})_w))=1.$\\
	\noindent $ (2):$ If statement $ (1) $ is not hold. Then by Theorem 3.4, $ \nullity(\mathcal{DL}^{\max}_<(\Phi))=0$ if and only if $ \Phi $ is unbalanced.
\end{proof}

\begin{theorem}\label{th3.6}
	For a $ \mathbb{T} $-gain graph $\Phi$ with a vertex ordering $ < $, one of the following statement holds.
	\begin{enumerate}
		\item [(1)] $ \nullity(\mathcal{DL}^{\min}_<(\Phi))=1$ if and only if $ \Phi $ is balanced.
		\item [(2)] $ \nullity(\mathcal{DL}^{\min}_<(\Phi))=0$ if and only if $ \Phi $ is unbalanced.
	\end{enumerate}
\end{theorem}

\begin{corollary}
	Let $\Phi$ be a $\mathbb{T} $-gain graph with a vertex ordering $ < $, Then the following hold:
	\begin{enumerate}
		\item [(1)] Either $ \nullity(\mathcal{DL}^{\max}_<(\Phi))=1$ or $\nullity(\mathcal{DL}^{\max}_<(\Phi))=0.$
		\item [(2)]  Either $ \nullity(\mathcal{DL}^{\min}_<(\Phi))=1$ or $\nullity(\mathcal{DL}^{\min}_<(\Phi))=0.$
	\end{enumerate} 
\end{corollary}
Let us present the following equivalence. 
\begin{theorem}
Let $ \Phi=(G, \varphi) $ be a $ \mathbb{T} $-gain graph. Then the following are equivalent:
\begin{enumerate}
	\item [(1)] $ \Phi $ is balanced.
	\item [(2)] $ \Phi $ is ordering independent and $ \nullity(\mathcal{DL}^{\max}(\Phi))=1.$
	\item [(3)] $ \Phi $ is ordering independent and $ \nullity(\mathcal{DL}^{\min}(\Phi))=1.$
	\item [(4)] $ \Phi $ is distance compatible and $ \nullity(\mathcal{DL}(\Phi))=1.$
	\item [(5)] $ \det \mathcal{DL}^{\max}(\Phi)=0 $
	\item[(6)] $ \det \mathcal{DL}^{\min}(\Phi)=0 $
	\item [(7)] $ \det \mathcal{DL}(\Phi)=0 $
\end{enumerate}
\end{theorem}

\section{Bounds for gain distance Laplacian spectral radius }

In this section we establish an upper bound  of spectral radius for gain distance Laplacian matrices in terms of the Winner index and number of vertices. Then we obtain some lower bounds of spectral radius in terms of distance signless Laplacian matrix of underlying graph and vertex transmission and characterize the equalities. Transmission of a vertex $ v_i $ of a connected graph $ G $ is defined as $ Tr(v_i):=\sum\limits_{j=1}^{n} d(v_i,v_j)$. The \emph{winner index} of a connected graph $ G $ is defined as $ W(G):=\sum\limits_{i>j}d(v_i,v_j) $. 
Let us present the following lower bound.
\begin{theorem}\label{th5.1}
	Let $ \Phi=(G, \varphi) $ be a connected  $ \mathbb{T} $-gain graph of $ n $ vertices with at least one edge and associated with a vertex ordering $ < $. Then 
	\begin{enumerate}
		\item[(i)] $ \rho(\mathcal{DL}^{\max}_{<}(\Phi)) \geq \frac{2W(G)}{n-1} $
		\item [(ii)] $ \rho(\mathcal{DL}^{\min}_{<}(\Phi)) \geq \frac{2W(G)}{n-1} $.
	\end{enumerate} 
Both the inequalities are sharp.
\end{theorem}
\begin{proof}(i) Let $ \Phi=(G, \varphi) $ be a connected $ \mathbb{T} $-gain graph of $ n $ vertices with a vertex ordering $ < $. Then by Proposition \ref{prop.4.1}, $ \mathcal{DL}^{\max}_<(\Phi) $ is positive semi-definite. Using Cholesky decomposition, there is a lower triangular matrix $ M $ with non negative diagonal such that \begin{equation}\label{eq1}
		\mathcal{DL}^{\max}_<(\Phi) =MM^*
	\end{equation}
	Let $ l_{ij} $ and $ m_{ij} $ be the $ (i,j) $th entry of $ \mathcal{DL}^{\max}_<(\Phi) $ and $ M $, respectively. Let $ C(i):=\sum\limits_{j=1}^{n}d(v_i,v_j) $, for $ i=1,2, \cdots, n $. Without loss of generality, consider $ C(1)=\max\limits_{i} C(i) $. Comparing $ (1,1) $th entry of matrices from both side of the equation \eqref{eq1}, we have $ l_{11}=C(1)=|m_{11}|^{2} $. Also comparing first column of matrices in equation \eqref{eq1}, $ l_{i1}=m_{i1}\overline{m}_{11} $, for $ i=1,2, \cdots,n $. Thus $ |l_{i1}|^{2}=|m_{i1}|^{2}|m_{11}|^{2} $, for $ i=1,2, \cdots,n $. Now the $ (1,1) $th entry of $ M^*M $ is $ \sum\limits_{i=1}^{n}|m_{i1}|^2 $. Then 
	\begin{align*}
		\sum\limits_{i=1}^{n}|m_{i1}|^2 &=\frac{1}{|m_{11}|^2}\sum\limits_{i=1}^{n}|l_{i1}|^2\\
		&=\frac{1}{C(1)}\left( C(1)^2+d(v_1,v_2)^2+\cdots+d(v_1,v_n)^2\right)\\
		&=C(1)+\frac{d(v_1,v_2)^2+\cdots+d(v_1,v_n)^2}{d(v_1,v_2)+\cdots+d(v_1,v_n)}
	\end{align*}

By Cauchy-Schwartz inequality, $ \frac{d(v_1,v_2)^2+\cdots+d(v_1,v_n)^2}{d(v_1,v_2)+\cdots+d(v_1,v_n)} \geq \frac{d(v_1,v_2)+\cdots+d(v_1,v_n)}{n-1}=\frac{C(1)}{n-1}$. Since $ \frac{\sum\limits_{i=1}^{n}C(i)}{n}=\frac{2W(G)}{n}$, so $ C(1)\geq \frac{2W(G)}{n}$. Therefore, 
\begin{align*}
	\sum\limits_{i=1}^{n}|m_{i1}|^2\geq \frac{2W(G)}{n-1}.
\end{align*}

Thus $ \rho(\mathcal{DL}^{\max}_<(\Phi))\geq \frac{2W(G)}{n-1} $.
Let us consider a balanced complete graph $ \Phi=(K_n,1) $. Then $ \mathcal{DL}^{\max}_{<}(\Phi)= nI-J_n$, where $ I $ is the identity matrix and $ J_n $ is a $ n \times n $ matrix with all entries are $ 1 $. Thus $ \rho(\mathcal{DL}^{\max}_{<}(\Phi))=n-1=\frac{2W(n)}{n-1}$. Equality occurs. Thus the inequality (i) is sharp. Similarly, $ \rho(\mathcal{DL}^{\min}_<(\Phi))\geq \frac{2W(G)}{n-1} $ and also the inequality is sharp.
\end{proof}

Let us present an example to show if $ \Phi $ is not distance compatible then $ \Phi \sim \Psi $ need not revel that $ \spec(\mathcal{DL}^{\max}_<(\Phi)) =\spec(\mathcal{DL}^{\max}_<(\Psi)) $.
\begin{figure} [!htb]
	\begin{center}
		\includegraphics[scale= 0.45]{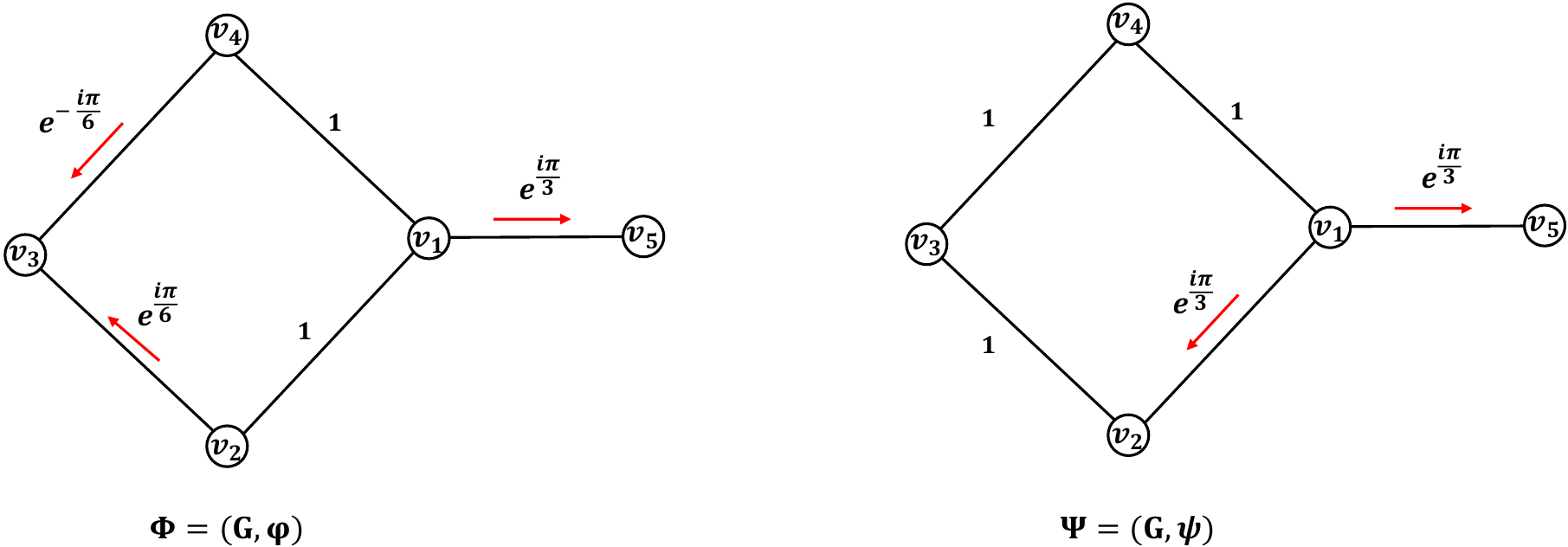}
		\caption{ Two switching equivalent $ \mathbb{T} $-gain graphs $ \Phi=(G, \varphi) $ and $ \Psi=(G, \psi) $} \label{fig1}
	\end{center}
\end{figure}

\begin{example}\label{ex4.1}{\rm
	Let $ \Phi=(G, \varphi) $ and $\Psi=(G, \psi)$ be two switching equivalent $ \mathbb{T} $-gain graphs with standard vertex ordering $ < $ (see Figure \ref{fig1}). Since $ \mathcal{D}^{\max}_{<}(\Phi) \ne \mathcal{D}^{\max}_{<_r}(\Phi) $, so $ \Phi $ is not distance compatible. Now
	\begin{equation*}
		\mathcal{DL}^{\max}_{<}(\Phi)=\left[\begin{array}{ccccc}
			5 & -1 & -2e^{\frac{i\pi}{6}} & -1 & -e^{\frac{i\pi}{3}}\\
			-1 & 6 & -e^{\frac{i\pi}{6}} & -2 & -2e^{\frac{i\pi}{3}}\\
			-2e^{-\frac{i\pi}{6}} & -e^{-\frac{i\pi}{6}} & 7 & -e^{\frac{i\pi}{6}} & -3e^{\frac{i\pi}{6}}\\
			-1 & -2 & -e^{-\frac{i\pi}{6}} & 6 & -2e^{\frac{i\pi}{3}}\\
			-e^{-\frac{i\pi}{3}} & -2e^{-\frac{i\pi}{3}} & -3e^{-\frac{i\pi}{6}} & -2e^{-\frac{i\pi}{3}} & 8
		\end{array} \right].
	\end{equation*}
and 	
	\begin{equation*}
	\mathcal{DL}^{\max}_{<}(\Psi)=\left[\begin{array}{ccccc}
		5 & -e^{\frac{i\pi}{3}}  & -2e^{\frac{i\pi}{3}} & -1 & -e^{\frac{i\pi}{3}}\\
		-e^{-\frac{i\pi}{3}} & 6 & -1 & -2 & -2\\
		-2e^{-\frac{i\pi}{3}} & -1 & 7 & -1 & -3e^{\frac{i\pi}{3}}\\
		-1 & -2 & -1 & 6 & -2e^{\frac{i\pi}{3}}\\
		-e^{-\frac{i\pi}{3}} & -2 & -3e^{-\frac{i\pi}{3}} & -2e^{-\frac{i\pi}{3}} & 8
	\end{array} \right].
\end{equation*}

Then $\spec(\mathcal{DL}^{\max}_{<}(\Phi))\ne 	\spec(\mathcal{DL}^{\max}_{<}(\Psi)) $.}	
\end{example}

However, if a $ \mathbb{T} $-gain graph is distance compatible then $ \Phi \sim \Psi$ implies that their distance Laplacian spectrum will be same.

\begin{theorem}\label{th4.8}
	Let $ \Phi=(G, \varphi)$ be a distance compatible connected $ \mathbb{T}$-gain graph. If $ \Phi \sim \Psi $ then $ \Psi $ is distance compatible and $ \spec(\mathcal{DL}(\Phi))=\spec(\mathcal{DL}(\Psi)) $. 
\end{theorem}
\begin{proof}Let $\Phi=(G, \varphi)$ be a distance compatible connected $ \mathbb{T} $-gain graph with vertex set $ V(G)=\{ v_1, v_2, \cdots, v_n\} $. Then any oriented shortest path $ v_iPv_j $ from $ v_i $ to $ v_j $ having the same gain in $ \Phi $. Also $ \Phi \sim \Psi $, so there is a switching function $ \zeta $ such that $ \psi(v_iPv_j)=\zeta(v_i)^{-1}\varphi(v_iPv_j)\zeta(v_j)$. Therefore, for any oriented shortest path from $ v_i$ to $ v_j$ having the unique gain $ \psi(v_iPv_j) $. Thus $ \Psi $ is distance compatible. Let $ d_\varphi(v_i, v_j) $ and $ d_\psi(v_i,v_j) $ be the unique distance from $ v_i $ to $ v_j $ in $ \Phi $ and $ \Psi $, respectively. Then $ d_\psi(v_i,v_j)= \zeta(v_i)^{-1}d_\varphi(v_i,v_j)\zeta(v_j)$. Let $ U=diag(\zeta(v_1), \zeta(v_2), \cdots, \zeta(v_n))$. Then $ \mathcal{D}(\Psi)=U^*\mathcal{D}(\Phi)U $ and hence $ \mathcal{DL}(\Phi)=U\mathcal{DL}(\Psi)U^* $.
\end{proof}

It is to be observe that if a $ \mathbb{T}$-gain graph $ \Phi=(G, \varphi) $ is distance compatible, then for any pair of vertices $ (v_i, v_j) $, there is a unique auxiliary gain. That is, $ \varphi^{<}_{\max}(v_i,v_j)=\varphi^{<_r}_{\max}(v_i,v_j)=\varphi^{<}_{\min}(v_i,v_j)=\varphi^{<_r}_{\min}(v_i,v_j)$, for any vertex ordering $ < $. Then the complete $ \mathbb{T} $-gain graph associated with $ \mathcal{D}(\Phi) $ is $ K^\mathcal{D}(\Phi):=(K_n, \psi) $, where $ \psi(\vec{e}_{i,j}):=\varphi^{<}_{\max}(v_i,v_j)=\varphi^{<}_{\min}(v_i,v_j) $, for all $ v_i, v_j\in V(G)$. A $ \mathbb{T} $-gain graph $ \Phi=(G, \varphi) $ is called \emph{anti-balanced} if $ -\Phi=(G, -\varphi) $ is balanced. Now we established the following upper bound.

\begin{theorem}\label{th4.9}
	Let $ \Phi=(G, \varphi) $ be a connected distance compatible $ \mathbb{T} $-gain graph. Then $ \rho(\mathcal{DL}(\Phi)) \leq \rho(\mathcal{DQ}(G)) $. Equality occur if and only if $ \Phi \sim (K_n, -1) $.
\end{theorem}

\begin{proof}
Let $ \Phi=(G, \varphi) $ be a connected $ \mathbb{T} $-gain graph on $ n $ vertices. Then $ |\mathcal{DL}(\Phi)|=\mathcal{DQ}(G) $. Since $ |\mathcal{DL}(\Phi)|\leq \mathcal{DQ}(G) $, so by Theorem \ref{th4}(1), 
\begin{equation}\label{eq1}
	\rho(\mathcal{DL}(\Phi)) \leq \rho(\mathcal{DQ}(G))
\end{equation}
By Theorem \ref{th4}(2), Equality occur in \eqref{eq1} if and only if $ \mathcal{DL}(\Phi)=e^{i\theta}U(\mathcal{DQ}(G))U^* $, where $ |U|=I $.  Since $ \mathcal{DL}(\Phi) $ is positive semi-definite, so $ \rho(\mathcal{DL}(\Phi)) $ is the largest eigenvalue of $ \mathcal{DL}(\Phi) $. Thus $ e^{i\theta}=1$. Therefore $ \rho(\mathcal{DL}(\Phi))= \rho(\mathcal{DQ}(G)) $ if and only if \begin{equation}\label{eq3}
	\mathcal{DL}(\Phi)=U(\mathcal{DQ}(G))U^*
\end{equation}
Now form $ \eqref{eq3} $, 
\begin{align*}
	\mathcal{DL}(\Phi)=U(\mathcal{DQ}(G))U^* & \Leftrightarrow  Tr(G)-\mathcal{D}(\Phi)=U\left(Tr(G)+\mathcal{D}(G)\right)U^*\\
	 & \Leftrightarrow -\mathcal{D}(\Phi)=U\mathcal{D}(G)U^*\\
     &\Leftrightarrow  -A(K^\mathcal{D}(\Phi))\circ \mathcal{D}(G)=U\left( A(K_n)\circ \mathcal{D}(G)\right)U^*\\
	 & \Leftrightarrow-A(K^\mathcal{D}(\Phi))\circ \mathcal{D}(G)=\left( UA(K_n)U^*\right)\circ \mathcal{D}(G)
\end{align*}
Then there is a matrix $ S $ such that $ D(G)\circ S=A(K_n) $.
Thus $$ -A(K^\mathcal{D}(\Phi))= UA(K_n)U^*$$. That is, $ K^\mathcal{D}(\Phi) $ is anti balanced.

Suppose $ \rho(\mathcal{DL}(\Phi))= \rho(\mathcal{DQ}(G)) $. \\
\textbf{Claim:} $ G $ is complete.\\
Let $ K^{\mathcal{D}}(\Phi)=(K_n, \psi) $, where $ \psi(\vec{e}_{i,j})=\varphi(v_iPv_j) $ for any shortest path $ v_iPv_j $. If possible let $ G $ is not complete. Then there is a pair of non adjacent vertices $ v_s, v_t $ in $ G $ such that $ d(v_s,v_t)=2 $. Let $ v_s \sim v_r \sim v_t $ be a path in $ G $. Then by the definition of $ K^\mathcal{D}(\Phi) $, $ \psi(C_3)=1 $, where $ C_3: v_s \sim v_r \sim v_t\sim v_s $. Thus $ -\psi(C_3)=-1 $. Which contradict that $ K^\mathcal{D}(\Phi) $ is anti balanced. Hence $ G $ is complete.

Therefore $ K^\mathcal{D}(\Phi)=\Phi $. Thus $ \Phi $ is anti balanced complete $ \mathbb{T} $-gain graph. Conversely, let $ \Phi \sim (K_n, -1) $. Then by Theorem \ref{th4.8}, $ \spec(\mathcal{DL}(\Phi)) =\spec(\mathcal{DL}(K_n,-1))$. Now $ \mathcal{DQ}(K_n)=Tr(K_n)+\mathcal{D}(K_n)=\mathcal{DL}(K_n,-1) $. Therefore $ \rho(\mathcal{DL}(\Phi))=\rho(\mathcal{DQ}(K_n)) $.
\end{proof}

Now we present an upper bound for spectral radius of gain distance Laplacian matrix in terms of the maximum vertex transmission. 

\begin{theorem} \label{th5.4}
Let $\Phi=(G, \varphi) $ be a connected distance compatible $ \mathbb{T} $-gain graph. Then 
\begin{equation}\label{eq5}
	\rho(\mathcal{DL}(\Phi))\leq 2 \max_{v\in V(G)}Tr(v).
\end{equation}
Equality occur if and only if $ \Phi \sim (K_n, -1) $.	
\end{theorem}

\begin{proof}
	Let $ \Phi=(G, \varphi) $ be a distance compatible $ \mathbb{T} $-gain graph. Then $ \mathcal{DQ}(G)=Tr(G)+D(G) $. Let $ \lambda $ be the largest eigenvalue of $\mathcal{DQ}(G)$ associated with an eigenvalue $ x $. Therefore by Perron-Frobenius theorem $ \rho(\mathcal{DQ}(G))=\lambda$ and $ x>0 $. Thus by Lemma \ref{lm2.1},\break $ \rho(\mathcal{DQ}(G)) \leq 2\max\limits_{v\in V(G)}Tr(v) $. Theorem \ref{th4.9} revels that $ \rho(\mathcal{DL}(\Phi)) \leq 2\max\limits_{v\in V(G)}Tr(v)$. Suppose  equality occur in \eqref{eq5}. By Theorem \ref{th4.9}, $ \Phi \sim (K_n, -1) $. Conversely let $ \Phi \sim (K_n, -1) $. Then  $\rho(\mathcal{DL}(\Phi))=2(n-1)=2\max\limits_{v\in V(G)}Tr(v)$.  
\end{proof}

\begin{remark}
	Theorem \ref{th5.1}, Theorem \ref{th5.4} and Theorem \ref{th4.9} hold for signed graphs.
\end{remark}
 \section*{Acknowledgments}
	Aniruddha Samanta expresses thanks to the National Board for Higher Mathematics (NBHM), Department of Atomic Energy, India, for providing financial support in the form of an NBHM Post-doctoral Fellowship (Sanction Order No. 0204/21/2023/R\&D-II/10038). The first author also acknowledges excellent working conditions in the Theoretical Statistics and Mathematics Unit, Indian Statistical Institute Kolkata.

\mbox{}


\end{document}